\documentclass[12pt]{amsart}%
\usepackage{amsmath}
\usepackage{amssymb}
\usepackage{amsthm}
\usepackage{amscd}
\usepackage[dvipdfmx]{graphicx}
\usepackage[mathscr]{eucal}
\usepackage{bm}
\usepackage{xcolor}
\makeindex
\newtheorem{Thm}{Theorem}[section]
\theoremstyle{definition}
\newtheorem{Theorem}[Thm]{Theorem}
\newtheorem{Lemma}[Thm]{Lemma}
\newtheorem{Corollary}[Thm]{Corollary}
\newtheorem{Proposition}[Thm]{Proposition}

\theoremstyle{remark}
\newtheorem{Remark}{Remark}

\font\ym=msbm10

\newcommand{\R}{\text{\ym R}}

\newcommand{\C}{\text{\ym C}}

\newcommand{\sA}{\mathscr A}
\newcommand{\sB}{\mathscr B}
\newcommand{\sC}{\mathscr C}

\newcommand{\sH}{\mathscr H}

\newcommand{\sR}{\mathscr R}

\newcommand{\Hom}{\hbox{\rm Hom}}
\newcommand{\End}{\hbox{\rm End}}

\title[radial representation]{On the primitive ideal space of radial representations of free groups}
\author{Yamagami Shigeru}
\begin{document}
% Haagerup positive definite functions on free groups.
% To Nobuko with our deep gratitudes. 
\maketitle
\begin{center}
Graduate School of Mathematics
\end{center}
\begin{center}
Nagoya University 
\end{center}
\begin{center} 
Nagoya, 464-8602, JAPAN
\end{center} 



% Keywords: Haagerup function, Stieltjes inversion formula, polynomial hypergroup, moment problem,
% analytic linear functional, 
% 46L99 (operator algebra), 46J25(commutative Banach algebra), 43A99(abstract harmonic analysis) 
% Related to radial functions on free groups, we focus on certain polynomial hypergroups and
% work out spectral analysis with the help of the Stieltjes transform of their analytic functionals.
\begin{abstract}
  Radial representations of finitely generated free groups are studied. The associated C*-algebra is located
  between the reduced and full group C*-algebras and its primitive ideal space is described concretely as a topological space.
%
  % The Plancherel formula for the standard trace on a finitely generated non-commutative free group
  % is generalized to disintegration of positive definite functions
  % of Haagerup in the framework of radial module. 
  % The radial module in turn gives the radial free group C*-algebra which is located
  % between the reduced and full group C*-algebra.
  % We sahll then describe the primitive ideal space of the radial group C*-algebra. 
\end{abstract}

\section*{Introduction}
Spherical functions on finitely generated free groups have been investigated from various points of view.
Among them, fundamental is the framework of Poisson boundaries, which allows us to take analogies with the case of
semisimple Lie groups. In fact, studies of spherical functions go more or less around it and, as their basic properties,
irreducibility and inequivalence of the associated spherical representations are established under this background. 

Other than spherical representations, a series of non-irreducible representations are 
associated with radial functions.
Within that class interesting are positive definite functions of Haagerup, which include the standard trace as a limit case.
For non-irreducible representations associated with positive definite radial functions, a generalization of Plancherel formula
is described in \cite{Y} via spectral decomposition of radial functions based on the radial algebra which is commutative.
Notice here that, different from the semisimple case, non-commutative free groups are not type I and the uniqueness of
decomposition into irreducible representations breaks down.

In other words, to obtain a Plancherel formula, we need to specify a maximal commutative subalgebra first
and the above-mentioned Plancherel formula is based on the algebra of radial functions. 
As an instance, the standard trace belongs to the radial class and its Plancherel measure turns out to be the so-called
Kesten measure, which is supported by the regular spectrum and equivalent to the Lebesgue measure. 

In this paper, we look into the primitive ideal space of radial representations and show that it is given by primitive ideals of spherical
representations. Furthermore, its hull-kernel topology is described explicitly in terms of the spectral parameter
which distinguishes equivalence classes of spherical representations.

During the process of description, we also see that 
% Points in the present investigation are (0) an explicit formula of the spectral decomposition of
% positive definite functions of Haagerup is combined with the radial bimodule
 the kernel of the regular representation is represented by a CCR algebra
 under the universal radial representation and a complete description of
 the primitive ideal space of the radial representation is obtained in terms of primitive ideals of this CCR algebra.

 The author is grateful to Tomohiro Hayashi for many discussions on the subject. 

\section{Background Review}
 We shall work with the countable discrete group $G$ generated freely by a finite set $\{ s_1,\dots,s_l\}$ and % deal with
 representations of $G$ by bounded linear operators on a Hilbert space $\sH$.
 By freeness, such a representation is specified by assigning a finite family $(S_i)_{1 \leq i \leq l}$ of bounded invertible operators
 on $\sH$ so that a unitary representation corresponds to a family of unitary operators.

 Let $\C G = \sum_{g \in G} \C g$ be the algebraic group algebra of $G$, which is naturally identified with the convolution algebra
 of functions on $G$ having finite supports. The completion
 $\ell^1(G)$ of $\C G$ with respect to the $\ell^1$-norm consists of summable sequences labeled by
 elements in $G$ so that the convolution product makes $\ell^1(G)$ a unital Banach algebra.

 The *-operation $g \mapsto g^* = g^{-1}$ on $\C G$ is obviously extended to $\ell^1(G)$ isometrically.
 Thus unitary representations of $G$ are in one-to-one correspondence with *-representations of $\C G$ or $\ell^1(G)$ on Hilbert spaces.
 The associated universal C*-algebra is called the group C*-algebra of $G$ and denoted by $C^*(G)$.
 The \textbf{parity automorphism} $\varpi$ of $C^*(G)$ is defined by $\varpi(g) = (-1)^{|g|}g$,
 where $|g|$ denotes the length of $g \in G$ with respect
 to the generator $\{ s_1,\dots,s_l\}$. 

 Among *-representations, the regular representation on $\ell^2(G)$ plays a central role in what follows:
 the left regular representation is denoted by $\lambda$.
 The image $\lambda(C^*(G))$, also denoted by  $C^*_r(G)$, is the reduced group C*-algebra and known to be simple (\cite{Po75}). 
 Recall that the regular representation is related with the standard trace $\tau$ on $C^*(G)$ specified by
 $\tau(g) = \delta_{g,e}$ for $g \in \C G \subset C^*(G)$: Let $\tau^{1/2}$ be the GNS vector of $\tau$.
 Then, for $f = \sum_{g \in G} f(g) g \in \C G$, 
 $f\tau^{1/2} = \tau^{1/2} f$ is identified with the function $f \in \ell^2(G)$ so that
 $\lambda(f)$ is realized by left multiplication of $f$ on the GNS representation space of $\tau$. 

 The parity automorphism $\varpi$ is implemented on $C^*_r(G)$ by the parity operator $\Pi$ on $\ell^2(G)$ defined by
 $\Pi(g\tau^{1/2}) = \varpi(g)\tau^{1/2}$. 

 % From here on,
We shall now review (more or less well-known) relevant facts on spherical representations of free groups together with some comments.

 Spherical functions on free groups are introduced and studied as counterparts of the semisimple case
 (see \cite{Ca, FP, MZ83, PS} for example).

 A function of $g \in G$ is said to be \textbf{radial} if it depends only through the length $|g|$.
 Radial functions supported by finite sets constitute a commutative *-subalgebra $\sA$ of $\C G$.
 Let $G_n = \{ g \in G; |g| = n\}$ with its indicator function denoted by $1_{G_n}$. Note that the number of elements in $G_n$ is
 $|G_n| = 2l(2l-1)^{n-1}$ ($n \geq 1$). 
 As a linear base of $\sA$, elementary radial functions are introduced by
 \[
   h_n = \frac{1}{|G_n|} 1_{G_n}
   \quad
   (n = 0,1,2,\dots),  
 \]
 which are hermitian $h_n^* = h_n$ with $h_0$ the unit element of $\C G$ and fulfill the recurrence relation
 \[
   h_1h_n = rh_{n-1} + (1-r) h_{n+1}
   \quad
   (n \geq 1)
 \]
 with $r = 1/2l$ satisfying $0 < r \leq 1/2$.
 
 If we introduce a polynomial sequence $p_n(t)$ of indeterminate $t$ by
 $p_0(t) = 1$, $p_1(t) = t$ and
 $t p_n(t) = r p_{n-1}(t) + (1-r) p_{n+1}(t)$ ($n \geq 1$), then $h_n = p_n(h_1)$ and
 \[
   \sum_{n=0}^\infty p_n(t) z^n = \frac{1-r - rtz}{1-r - tz + rz^2}.
 \]

 The spectrum of $\lambda(h_1)$ is exactly
 \[
   \sigma_r = [-2\sqrt{r(1-r)},2\sqrt{r(1-r)}]
 \]
 (called the \textbf{regular spectrum}) without eigenvalues inside 
 (\cite{Ke, Co}).
 
 % A \textbf{spherical function} is by definition a radial function $\phi$ satisfying $\phi(e) = 1$ and 
 % $\phi h_1 = c \phi$ for some $c \in \C$.
 % A spherical function $\phi$ turns out to be in one-to-one correspondence with a multiplicative functional $\varphi$ on $\sA$ by
 % the algebraic relation $\varphi(a) = \tau(a\phi)$ ($a \in \sA$) with the parameter $c$ related by 
 % $\varphi(h_n) = p_n(c)$ ($n \geq 0$). The eigenvalue $c$ is then referred to as the \textbf{spectral parameter} of $\phi$ or $\varphi$.

 Let $E: \C G \to \sA$ be the averaging map defined linearly by $E(g) = h_{|g|}$ ($g \in G \subset \C G$), 
which satisfies algebraic properties of conditional expectation:
$E(a) = a$, $E(f^*) = E(f)^*$ and $E(af) = aE(f)$ for $a \in \sA$ and $f \in \C G$.
A radial function $\phi$ is then in one-to-one correspondence with a linear functional $\varphi$ of $\sA$ by the relation
$\phi(g) = \varphi\circ E(g)$ ($g \in G$).
The averaging map $E$ is also characterized by the equality $\tau(f\phi) = \tau(E(f)\phi)$ for $f \in \C G$ and
a radial function $\phi$.

Given a complex number $c \in \C$, let $\delta_c$ be the multiplicative linear functional of $\sA$ evaluated at $c$, i.e.,
$\delta_c(h_n) = p_n(c)$ ($n \geq 0$). 
When $\varphi = \delta_c$, the associated radial function $\varepsilon_c$ is called a \textbf{spherical function} with $c$ referred to
as the \textbf{spectral parameter}. 
% is by definition a radial function $\phi$ satisfying $\phi(e) = 1$ and $\phi h_1 = c \phi$ for some $c \in \C$,
% which turns out to be of the form $\phi(g) = \varphi\circ E(g)$ with $\varphi$ a multiplicative
% functional on $\sA$ specified by $\varphi(h_n) = p_n(c)$ ($n \geq 0$).
%  A spherical function $\phi$ turns out to be in one-to-one correspondence with a multiplicative functional $\varphi$ on $\sA$ by
%  the algebraic relation $\varphi(a) = \tau(a\phi)$ ($a \in \sA$) with the parameter $c$ related by 
%  $\varphi(h_n) = p_n(c)$ ($n \geq 0$).
%  The eigenvalue $c$ is then referred to as the \textbf{spectral parameter} of $\phi$ or $\varphi$.

% and the spherical function $\phi$ associated to a multiplicative
% functional $\varphi$ on $\sA$ is expressed by $\phi(g) = \varphi(E(g))$ ($g \in G$).

A multiplicative functional $\delta_c$ of spectral parameter $c \in \C$ is $\ell^1$-bounded if and only if
$c = a + ib$ ($a,b \in \R$) is in the elliptic disk
 \[
   % (\text{Re}\,c + \text{Im}\, c
   a^2 + \frac{b^2}{(1-2r)^2} \leq 1
 \]
 (when $r = 1/2$, this shrinks to the interval $[-1,1] \subset \R$) and $\delta_c$ is C*-bounded if and only if $-1 \leq c \leq 1$ (\cite{Ca, Py81}).
 Notice that C*-bounded multiplicative functionals are automatically *-preserving.

 With this spectral property of $c$ in hand, we see that a spherical function $\varepsilon_c$ of spectral parameter $c$ is positive definite
 if and only if $-1 \leq c \leq 1$.

 Let $A$ be the closure of $\sA$ in the full group C*-algebra $C^*(G)$.
 Then, by Haagerup (see \cite{S90, Y}),
 $E$ is extended to a conditional expectation (also denoted by $E$) of $C^*(G)$ onto $A \subset C^*(G)$. 

 % (see [S2] for related results and c.f.~also [Y]).

 Since $E$ preserves positivity, any positive functional $\varphi$ of $A$ induces a positive functional $\varphi\circ E$ of $C^*(G)$. 
 If we take an evaluation functional $\delta_t$ on $A$ at $t \in [-1,1]$ as a $\varphi$,
 it gives rise to a spherical state on $C^*(G)$,
 which is also denoted by $\varepsilon_t$, i.e., $\varepsilon_t = \delta_t\circ E$ and the spectrum of $A$ is identified with
 $[-1,1]$, i.e., $A \cong C([-1,1])$ ($C(K)$ standing for the continuous function algebra of a compact set $K$). 
 
 Remark that the spherical state $\varepsilon_t$ is the obvious extension of
 a positive definite spherical function of spectral parameter $-1 \leq t \leq 1$ 
 with $\epsilon_{\pm 1}(g) = (\pm 1)^{|g|}$ multiplicative on $G$, i.e., the trivial/parity character (one-dimensional representation) of $G$.

 The GNS-representation of $\varepsilon_t$ is called a \textbf{spherical representation}.
 
 \begin{Theorem}[\cite{FP, MZ87, S90}] % (see also [S2], [MZ87])
 Spherical representations are irreducible and mutually disjoint for different $t \in [-1,1]$. 
 \end{Theorem}

 Spherical representations of spectral parameter in $\sigma_r$ appear as irreducible components
 of the regular representation in a form of Plancherel formula 
  and are referred to as being \textbf{principal}, whereas ones parametrized by $[-1,1] \setminus \sigma_r$ are said
  to be \textbf{complementary}. %  Behaviour of spherical representations is different dependeng

 As it will be reviewed below, these series of representations (except for the residual values $t = \pm 1$) are realized as
 an analytic family of representations on a single Hilbert space.
 
 We now introduce another analytic family $(\lambda_z)$ of representations on $\ell^2(G)$ due to Pytlik and Szwarc.
 (The original notation is changed to $\lambda_z$ in view of the fact that this is a deformation of $\lambda$.)
 Here are basic properties: $(\lambda_z)$ is a family of bounded representations of $G$ on $\ell^2(G)$ parametrized by
 a complex number $z$ satisfying $z^2 \not\in (1,\infty)$.

 \begin{Theorem}[\cite{PS}]~ 
 \begin{enumerate}
 \item
   For each $g \in G$, $\lambda_z(g)$ is a finite-rank perturbation of $\lambda(g)$: 
   there is a finite dimensional subspace $L_g$ of $\ell^2(G)$ satisfying
   $\lambda_z(g) = \lambda(g)$ on $L_g^\perp$ for each $z$.
 \item
   For any $g \in G$, $\lambda_z(g)$ is continuous in $z$ and holomorphic on $z \in \C \setminus ((-\infty,-1] \sqcup [1,\infty))$.
 \item
   The equality $\lambda_z(g)^* = \lambda_{\overline{z}}(g^{-1})$ holds for any $g \in G$ and any $z$.
   Consequently $\lambda_z$ is unitary if $z$ is real, i.e., if $z \in [-1,1]$.
 \item
   For $z^2 \not\in [1,\infty)$, $\tau^{1/2}$ is a cyclic vector of $\lambda_z$ and satisfies
   \[
     (\tau^{1/2}|\lambda_z(g)\tau^{1/2}) = z^{|g|}
     \quad
     (g \in G). 
   \]
   In particular, the function $z^{|g|}$ is positive definite for $-1 \leq z \leq 1$ (\cite{Ha}) and $\lambda_0 = \lambda$. 
    \item 
  Limits $\lambda_{\pm 1} = \lim_{z \to \pm 1} \lambda_z$ are unitary representations of $G$ on $\ell^2(G)$ 
  satisfying
  \[
    \lambda_{\pm 1}(a) g\tau^{1/2} =
    \begin{cases}
      \pm \tau^{1/2} &(g = e)\\
      \mp a\tau^{1/2} &(g = a^{-1})\\
      ag\tau^{1/2} &\text{otherwise}
    \end{cases}
  \]
  for $a \in G_1$ and $g \in G$.
 \end{enumerate}
\end{Theorem}

\begin{Corollary}\label{ends}
  According to an orthogonal decomposition $\ell^2(G) = \C \tau^{1/2} \oplus \ell^2(G(1)) \oplus \dots \oplus \ell^2(G(l))$,
  where $G(i)$ consists of words whose right ends are in $\{ s_i,s_i^{-1}\}$,
 $\lambda_{\pm 1}$ is unitarily equivalent to a direct sum of $\lambda$ by multiplicity $l$ and the parity/trivial character of $G$. 
\end{Corollary}

\begin{proof} Let $G^* = G \setminus \{ e\}$. 
  Since $\lambda_{-1}(g)(G^*\tau^{1/2}) = G^*\tau^{1/2}$ for $g \in G$ by (v), $\lambda_{-1}(g)$ induces a free action of $G$ on $G^*$ with 
  $G^* = G(1) \sqcup \dots \sqcup G(l)$ the decomposition of $G^*$ into orbits. Thus $\lambda_{-1}$ on $\ell^2(G^*)$ is decomposed into
  the direct sum of subrepresentations on $\ell^2(G(i))$ ($1 \leq i \leq l$) so that the restriction of $\lambda_{-1}$ on $\ell^2(G(i))$
  is unitarily equivalent to the regular representation of $G$.

  Likewise $\lambda_{1}$ and the parity operator $\Pi$ leave $\ell^2(G(i))$ invariant so that
  $\Pi\lambda_1(x)\Pi = \lambda_{-1}(\varpi(x))$ ($x \in C^*(G)$).
  From the decomposition of $\lambda_{-1}$, this implies that $\lambda_1$ on $\ell^2(G(i))$ is unitarily equivalent to
  $\lambda\circ\varpi = \Pi \lambda \Pi \cong \lambda$ as well.   
\end{proof}

{\small
\begin{Remark}
  The above corollary is taken from \cite{PS} 2.4~Remark (3),
  where it is pointed out that $\lambda_{-1}$ on $\ell^2(G^*)$ is considered by J.~Cuntz to illustrate K-amenability of free groups. 
\end{Remark}}

The following supplements to \cite{PS} are extracted from \cite{S88,S90}
with the case of critical values $\upsilon = \pm\sqrt{r/(1-r)}$ added in \cite[Corollary~4.4]{Y}.
% with an addition for the critical values $u = \pm\sqrt{r/(1-r)}$ from \cite{Y}.

\begin{Theorem}\label{SS} Let $\lambda_\upsilon$ be a unitary representation, i.e., $\upsilon \in [-1,1]$. 
  \begin{enumerate}
  \item
    If $|\upsilon| \leq \sqrt{r/(1-r)}$, $\lambda_\upsilon$ is unitarily equivalent to $\lambda$.
  \item
    If $\sqrt{r/(1-r)} < |\upsilon| < 1$, $\lambda_\upsilon$ is unitarily equivalent to a direct sum of $\lambda$ and
    the spherical representation of spectral parameter % given by a weighted Joukowsky transform 
    \[
      c_r(\upsilon) = \frac{r}{\upsilon} + (1-r)\upsilon
    \]
   with the spectrum of $\lambda_\upsilon(h_1)$ equal to $\sigma_r \sqcup \{ c_r(\upsilon)\}$. 
  \end{enumerate}
\end{Theorem}

Given a positive functional $\varphi$ of $A$,
the GNS representation space $\sR_\varphi \equiv \overline{C^*(G) (\varphi\circ E)^{1/2}}$ ($(\varphi\circ E)^{1/2}$ being the GNS-cyclic vector)
turns out to be a $C^*(G)$-$A$ bimodule: The left action of $C^*(G)$ is just the (left) GNS-representation based on $\varphi\circ E$ and
the right action of $A$ is given by $(x(\varphi\circ E)^{1/2})a = (xa) (\varphi\circ E)^{1/2}$ ($x \in C^*(G)$, $a \in A$).

If $\psi$ is another positive functional of $A$,
% it is seen in \cite{Y} that
the space $\Hom(\sR_\varphi,\sR_\psi)$ of intertwiners is naturally isomorphic to $\Hom(L^2(\varphi),L^2(\psi))$ (\cite{Y}).
Here $L^2(\varphi)$ denotes the GNS representation space of $\varphi$ and $\Hom(L^2(\varphi),L^2(\psi))$ denotes the space of
intertwiners between $A$-modules $L^2(\varphi)$ and $L^2(\psi)$.

In view of $L^2(\varphi) = L^2(A)[\varphi]$ with $L^2(A)$ the standard space of the second dual W*-algebra $A^{**}$ and
$[\varphi]$ the support projection of $\varphi$ in $A^{**}$, we have
$\Hom(L^2(\varphi), L^2(\psi)) \cong A^{**} [\varphi][\psi]$.
Thus $\Hom(L^2(\varphi),L^2(\psi))$ is isomorphic to the $L^\infty$-space on $[-1,1]$ with respect to
the common measure class of $\varphi$ and $\psi$ in the Lebesgue decomposition. 

% we define the radial $C^*(G)$-$A$ bimodule $\sR_\varphi$ in such a way that
% $\sR_\varphi$ is the GNS-representation of $\varphi\circ E$ as a left $C^*(G)$-module
% with the GNS-cyclic vector denoted by $(\varphi\circ E)^{1/2}$
% and the right action of $A$ is given by $(x(\varphi\circ E)^{1/2})a = (xa) (\varphi\circ E)^{1/2}$ ($x \in C^*(G)$, $a \in A$).

When $\varphi = \delta_t$ with $\varepsilon_t = \varphi\circ E$ a spherical state on $C^*(G)$,
$\sR_\varphi$ is simply denoted by $\sR_t$. As observed in \cite{Y}, $(\sR_t)_{-1 \leq t \leq 1}$ is a continuous family of 
$C^*(G)$-$A$ bimodules and therefore it provides a Borel field structure so that the following holds. 

\begin{Theorem}[Plancherel Formula]
  Under the identification of $\varphi$ with the associated Radon measure $\varphi(dt)$ on the spectrum $[-1,1]$ of $A$,
  we have a natural isometric isomorphism between $\sR_\varphi$ and
  the direct integral $\oint_{[-1,1]} \sR_t\, \sqrt{\varphi(dt)}$ specified by the correspondence 
  \[
    \sR_\varphi \ni (\varphi\circ E)^{1/2} \mapsto \oint_{[-1,1]} \varepsilon_t^{1/2}\, \sqrt{\varphi(dt)}
    \in \oint_{[-1,1]} \sR_t\, \sqrt{\varphi(dt)}.
  \]
\end{Theorem}

For the positive definite function $\phi_\upsilon = \varphi_\upsilon\circ E$ ($-1 \leq \upsilon \leq 1$) of Haagerup,
where $\varphi_\upsilon(h_n) = \upsilon^n$ ($n \geq 0$), 
the accompanied Radon measure $\varphi_\upsilon(dt)$ takes the following form (\cite{Y}): 
\begin{enumerate}
\item
  If $|\upsilon| \leq \sqrt{r/(1-r)}$, $\varphi_\upsilon(dt)$ is supported by the interval $\sigma_r$ and of the form 
  \[
   \varphi_\upsilon(dt) = \frac{\upsilon^{-1} - \upsilon}{2\pi}
  \frac{\sqrt{4r(1-r) - t^2}}{(1-t^2)(c_r(\upsilon) - t)}\, dt.   
\]
\item
  If $\sqrt{r/(1-r)} < |\upsilon| \leq 1$, adding to the continuous measure in (i),
there appears an atomic measure of the form
\[
  \frac{1-c_r(\upsilon^2)}{1-c_r(\upsilon)^2} \delta(t-c_r(\upsilon)). 
\]
\end{enumerate}
In the special case $\upsilon=0$ for which $\phi_0 = \tau$, $\varphi_0(dt)$ is reduced to the Kesten measure
\[
  \varphi_0(dt) = \lim_{\upsilon \to 0} \varphi_\upsilon(dt) = \frac{1}{2\pi r} \frac{\sqrt{4r(1-r) - t^2}}{1-t^2}\, dt
\]
on $\sigma_r = [-2\sqrt{r(1-r)},2\sqrt{r(1-r)}]$ as stated in \cite{Ca} (see also \cite{Sa, Py81}). 

\bigskip
  Let $\pi$ be the GNS-representation of the spherical state $\varepsilon_s$ for a spectral parameter $s = 2\sqrt{r(1-r)}$
  and $\sH = \overline{C^*(G)\varsigma}$ be the representation space of $\pi$ with the GNS-vector denoted by $\varsigma = \varepsilon_s^{1/2}$.  
  Notice that the spectral parameter $s = 2\sqrt{r(1-r)}$ is critical in the sense that it is located at
  the boundary of principal and complementary series.
  It is also critical from the viewpoint of C*-completion relative to $\ell^p(G)$ (\cite{BG, Ok}). 
  % $\pi$ is called a critical representation. 
  
  As a deformation of $\pi$,
  an analytic family $(\pi_c)$ of bounded representations on $\sH$ is constructed in \cite{Py91}
  for $c \in \C$ satisfying $c^2 \not\in (1,\infty)$ % $\setminus((-\infty,-1) \sqcup (1,\infty))$
  in such a way that $\pi_s = \pi$
  and the following properties hold. 
  \begin{enumerate}
  \item
   For each $g \in G$, there is a finite dimensional subspace $\sH_g$ of $\sH$ so that
   $\pi_c(g) = \pi(g)$ on $\sH_g^\perp$ for any $c$, i.e., $\pi_c(g) - \pi(g)$ is a finite-rank operator. 
 \item
   For any $g \in G$, $\pi_c(g)$ is continuous in $c$ and holomorphic on $c \in \C \setminus ((-\infty,-1] \sqcup [1,\infty))$.
 \item
   We have $\pi_c(g)^* = \pi_{\overline{c}}(g^{-1})$ for $g \in G$ and any parameter $c$. Consequently $\pi_c$ is unitary if $c \in [-1,1]$.
 \item
 If $c \not= \pm 1$, % $c^2 \not\in [1,\infty)$,
 $\varsigma$ is a cyclic vector of $\pi_c$ and satisfies
 $(\varsigma|\pi_c(g)\varsigma) = \varepsilon_c(g) = \delta_c(h_{|g|})$ ($g \in G$).
\item
  For $c = \pm 1$, we have $\pi_{\pm 1}(g) \varsigma = \varepsilon_{\pm 1}(g) \varsigma$ ($g \in G$)
  but $\varsigma$ is not cyclic. 
 % \item
 %   $\pi_c(h_1) \varsigma = c \varsigma$ for any $c$. 
 \end{enumerate}

 From the property (iv), we see $(\varsigma|(\pi_c(h_1) - c)^2\varsigma) = 0$, i.e., $\pi_c(h_1)\varsigma = c\varsigma$ for $c \in (-1,1)$,
 and then, by analytic continuation on $c$, $\pi_c(h_1)\varsigma = c \varsigma$ for any $c$.
 Thus $\varsigma$ is an eigenvector of $\pi_c(h_1)$ of eigenvalue $c$.
 
 More is known on the spectrum of $\pi_c(h_1)$: 

 \begin{Theorem}[{\cite[Theorem 5]{S88}}, {\cite[Lemma 3]{Py91}}]\label{SP}
    Any spectral parameter $c \in \C$ in the closed elliptic disk 
  %  Let $c = \sqrt{r(1-r)}(\zeta + \zeta^{-1})$ with $\zeta$ satisfying
  % $\sqrt{r/(1-r)} < |\zeta| < 1$ or $1 < |\zeta| < \sqrt{(1-r)/r}$, i.e.,
  % $c$ belongs to an open elliptic disk with a slit
  % \[
  %   \{ x+iy: x^2 + \frac{y^2}{(1-2r)^2} \leq 1\} % \setminus [-2\sqrt{r(1-r)},2\sqrt{r(1-r)}]. 
  % \]
  of $\ell^1$-boundedness
  % $c = x+iy$ ($x,y \in \R$) belongs to an elliptic disk $(\text{Re}\,c)^2 + (\text{Im}\,c)^2/(1-2r)^2 \leq 1$ but not to $\sigma_r$.  
  % Then $c$
  is an eigenvalue of $\pi_c(h_1)$ with $\varsigma$ its eigenvector so that $\pi_c(h_1) \varsigma^\perp \subset \varsigma^\perp$ and
  the spectrum of $\pi_c(h_1)$ as a bounded operator on the reducing subspace $\varsigma^\perp \subset \sH$ is contained in
  the regular spectrum $\sigma_r = [-2\sqrt{r(1-r)},2\sqrt{r(1-r)}]$.
 \end{Theorem}

 {\small
 \begin{Remark}
  % There is similarity between \cite{S88} and \cite{Py91}. More precisely
   A common-space-realization $\Pi_z$ of the complementary component of $\lambda_z$ in \cite{S88} turns out to be extended to
   the region $z^2 \not\in [1,\infty)$ by an analytic continuation, which seems to be globally similar to $(\pi_z)$ in \cite{Py91}.
\end{Remark}}

%   Representations of principal series and complementary series behave very fidderently 
%    Behaviour of representations turns out to be different in these two series and
%   $\pm 2\sqrt{r(1-r)}$ look like critical points in phase transition.
  
% Let $\sH$ be the GNS-space of 
% Since the spectral measure of $\epsilon_t$ on $A$ 
% By the uniqueness of GNS-construction, $\lambda_t(h_1)$ 
% Since $\lambda_t(A) \tau^{1/2}$

\section{Primitive Ideal Space}
We begin with a universal construction of radial bimodules.
For two positive functionals $\varphi$ and $\psi$ of $A$ satisfying $\varphi \leq \psi$,
the spherical decomposition in the Plancherel formula enables us to define an isometric embedding by 
\[
  \oint_{[-1,1]} x\, \sqrt{\varphi(dt)} \mapsto \oint_{[-1,1]} x\, \sqrt{\frac{\varphi(dt)}{\psi(dt)}}\, \sqrt{\psi(dt)}
\]
($\varphi(dt)/\psi(dt)$ being the Radon-Nikodym derivative), which in turn is converted to
an isometric embedding of $\sR_\varphi$ into $\sR_\psi$ thanks to
the identity $(\varphi\circ E)^{1/2} = \sqrt{d\varphi/d\psi} (\psi\circ E)^{1/2}$
in the standard representation space $L^2(C^*(G))$ of the second dual $C^*(G)^{**}$ (see \cite{Yaa}).
% algebraic aspect 

By the way of this construction, embeddings $\sR_\varphi \to \sR_\psi$ give an inductive system of
$C^*(G)$-$A$ bimodules and we obtain the universal $C^*(G)$-$A$ bimodule $\sR$ as an inductive limit.
The image of the left action of $C^*(G)$ on $\sR$ is then called the radial C*-algebra of $G$ and denoted by $C^*_{\text{rad}}(G)$,
which is univerasl with respect to radial representations.
Our main concern is in describing the primitive ideal space of $C^*_{\text{rad}}(G)$.

In terms of this universal radial bimodule,
irreducibility of spherical representations is now rephrased in the following manner: 

\begin{Proposition}
  The von Neumann algebra $\text{End}({}_{C^*(G)} \sR_A)$ of intertwiners is isomorphic to $A^{**}$ through the right multiplication of $A^{**}$.
 % is generated by the right multiplication of $A$ on $\sR$. 
\end{Proposition}

\begin{proof}
 % Given $\varphi \in A^*_+$, let $\sR_\varphi = \overline{C^*(G) (\varphi\circ E)^{1/2}}$ be a subbimodule of $\sR$.
  By \cite{Y}, the space $\End(\sR_\varphi)$ of self-intertwiners is generated by the right action of $A$ and
  $\sR$ contains $L^2(A)$ as an $A$-$A$ subbimodule, whence the problem is reduced to showing that
  $\End(\sR)$ is generated by the right action of $A$ as a von Neumann algebra. 

  To see this, observe first that the projection $e_\varphi$ to the subspace $\sR_\varphi$ belongs to $\End(_{C^*(G)}\sR_A)$ and
  satisfies $\lim_{\varphi \uparrow \infty} e_\varphi = 1_\sR$ in the $\sigma$-strong operator topology.
  Let $T \in \End(_{C^*(G)}\sR_A)$. If its reduction $e_\varphi Te_\varphi \in \End(_{C^*(G)}\sH_A)$ commutes with $e_\varphi \End(\sR_A) e_\varphi$
  for any $\varphi \in A^*_+$, $a' \in \End(\sR_A)$ and $\xi, \eta \in \sR_\psi$ with $\psi \in A^*_+$ satisfy 
  \[
    (\xi|Te_\varphi a'\eta) = (\xi|e_\varphi T e_\varphi a' e_\varphi \eta)
    = (\xi|e_\varphi a' e_\varphi T e_\varphi \eta) = (\xi|a' e_\varphi T \eta)
    \quad
    (\varphi \geq \psi) 
  \]
  and then we have $(\xi|Ta'\eta) = (\xi|a'T\eta)$ by taking limit $\varphi \uparrow \infty$. 
  Since the inductive (algebraic) limit of $C^*(G) (\psi\circ E)^{1/2}$ for $\psi \uparrow \infty$ is dense in $\sR$, this implies
  $Ta' = a'T$, i.e., $T \in \End(\sR_A)'$.
  As a final step, apply the double commutant theorem of von Neumann. 
 %  Returning to the problem for $\sH$, 
 % %  Since $\{ g \epsilon_c^{1/2}\}_{c \in [-1,1]}$ are measurable families for $g \in G$ and generate
 % %  $\{ g a \epsilon_c^{1/2} = g \epsilon_c^{1/2} a\}_{c \in [-1,1]}$ 
 % % as measurable families for $g \in G$ and $a \in \sA$, we have a direct integral decomposition
 % %  \[
 % %    \sH = \oint \sH_c\, \mu(dc), \quad \sH_c = \overline{C^*(G) \epsilon_c^{1/2}}
 % %  \]
 % %  so that $L^\infty(\mu)$ acts from right as diagonalizable operators.
 % each $T \in \text{End}(\sH_A)$ is decomposed as
 %  \[
 %    T = \oint_{[-1,1]} T_t\, \varphi(dt)\quad
 %    \text{with}
 %    \ 
 %    T_t \in \cB(\sR_t). 
 %  \]
 %  If $T$ further commutes with countably many diagonalizable operators
 %  \[
 %    \oint_{[-1,1]} \pi_t(g) \, \varphi(dt)
 %  \]
 %  ($g \in G$), then
 %  $T_t \in \pi_t(G)'$ for almost any $t \in [-1,1]$.
 %  Since $\pi_t$ is irreducible, this means that $T_t$ are scalar operators
 %  for almost any $t$ and $T$ belongs to $L^\infty([-1,1])$. 
\end{proof}

\begin{Corollary}
  For a positive functional $\varphi$ of $A$, $\End(\sR_\varphi) \cong A^{**}[\varphi]$,
  where $[\varphi]$ denotes the support projection of $\varphi$ in $A^{**}$. 
\end{Corollary}

 % Let $\sigma_r = [-2\sqrt{r(1-r)},2\sqrt{r(1-r)}]$ be the regular spectrum.
We use the notation $\sC(\sH)$ to stand for the compact operator algebra on a Hilbert space $\sH$. 

 \begin{Proposition}\label{regularity}
   For a continuous function $f \in C([-1,1])$ vanishing on $\sigma_r \subset [-1,1]$, we have 
   % if this is the case,
   % if and only if $f$ vanishes on $\sigma_r$,
  $\pi_t(f(h_1)) = f(t) |\varsigma)(\varsigma|$ ($-1 \leq t \leq 1$)
  % Since $\pi_c$ is irreducible for $-1 < c < 1$,
  and the representation $\pi_t$ for $t \in (-1,1) \setminus \sigma_r$ is well-behaved 
  in the sense that $\sC(\sH) \subset \pi_t(C^*(G))$.
  % Here $\sC(\sH)$ denotes the compact operator algebra on $\sH$. 
 \end{Proposition}

 \begin{proof}
   By the spectral property in Theorem~\ref{SP},
   $\pi_t(f(h_1)) = f(t)|\varsigma)(\varsigma| = 0$ if $t \in \sigma_r$.
  
   For $t \in [-1,1] \setminus \sigma_r$,
   the spectrum of $\pi_t(f(h_1))$ is supported by the point $t$ with $|\varsigma)$ a unique eigenvector
   and hence $\pi_t(f(h_1)) = f(t) |\varsigma)(\varsigma|$ by functional calculus.
   
   Thus $\pi_t(C^*(G))$ contains a rank-one projection $|\varsigma)(\varsigma|$ and the irreducibility of $\pi_t$ for $-1 < t < 1$ implies 
   $\sC(\sH) \subset \pi_t(C^*(G))$.
 \end{proof}

 {\small
 \begin{Remark}
   Properties of complementary series representations concerning compact operator algebras are also pointed out in \cite{S90}.
   % point out witnessed
\end{Remark}}

 \begin{Proposition}\label{prim}
   Primitive ideals of $C^*(G)$ associated to pure states $\varepsilon_t$ are different for
   $t \in [-1,1] \setminus \sigma_r$
   and coincide with the kernel of the regular representation $\lambda$ for $t \in \sigma_r$.
 \end{Proposition}

 \begin{proof}
   Clearly characters $\varepsilon_{\pm 1}$ give rise to ideals different from those for $\varepsilon_t$ ($-1 < t < 1$)
   and we focus on irreducible representations $\pi_t$ ($-1 < t < 1$).
   
   If two primitive ideals coincide for spectral parameters in $(-1,1)$, 
   the associated C*-algebras $\pi_t(C^*(G))$'s are canonically isomorphic and give rise to the same spectrum of $\pi_t(h_1)$.
   In the case when one of $t$ is outside of $\sigma_r$,
   this necessitates the coincidence of $t$ by Theorem~\ref{SP}.

  Let $t \in \sigma_r$. If $x \in C^*(G)$ is in the kernel of $\lambda$,
  then the Plancherel formula ensures $\pi_{t'}(x^*x) = 0$ for almost any $t' \in \sigma_r$ with respect to the Kesten measure
  (being equivalent to the Lebesgue measure on $\sigma_r$) and
  we can find a sequence $t_n \in \sigma_r$ converging to $t$ so that $\pi_{t_n}(x^*x) = 0$.
  Since $\varepsilon_{t'}$ is weak*-continuous in $t'$, we have 
  \[
    \varepsilon_t(y^*x^*xy) = \lim_{n \to \infty} \varepsilon_{t_n}(y^*x^*xy) = 0
  \]
  for any $y \in C^*(G)$. Thus, by cyclicity of $\varsigma$, $\pi_t(x^*x) = 0$, i.e., $\ker\lambda \subset \ker \pi_t $
  with $\lambda(\ker\pi_t)$ a closed ideal of a simple C*-algebra $C^*_r(G)$, proving that
  $\ker\pi_t = \ker\lambda$. 
\end{proof}

% \begin{Remark}
% For $c \in [-2\sqrt{r(1-r)},2\sqrt{r(1-r)}]$, all the primitive ideals $\ker \pi_c$ coincide with the kernel of the regular representation. 
% \end{Remark}

\begin{Corollary}\label{singularity}~
  \begin{enumerate}
    \item
      We have $\pi_t(C^*(G)) \cap \sC(\sH) = \{ 0\}$ for $t \in \sigma_r$.
    \item
      For $f \in C([-1,1])$, the condition $f(h_1) \in \ker\pi$ is equivalent to $f|_{\sigma_r} = 0$. 
    \end{enumerate}
 \end{Corollary}

\begin{proof}
 % For $c \in \sigma_r$, the representation $\pi_c$ is singular; otherwise
  (i) If $\pi_t(C^*(G)) \cap \sC(\sH) \not= 0$, $\sC(\sH) \subset \pi_t(C^*(G))$ by irreducibility of $\pi_t$.
  % a consequence of Kadison's transitivity theorem, see /functional/convexset.pdf
  In view of $\ker \pi_{t'} = \ker \pi_t$ for $t' \in \sigma_r$, this implies $\pi_{t'} \cong \pi_t$ (\cite{Di} Corollary 4.1.10),
  % convexset.pdf Proposition 36. 
which contradicts with the disjointness of $\pi_t$ and $\pi_{t'}$ for $t \not= t'$ in $\sigma_r$.

(ii) From the kernel coincidence, the spectrum of $\pi(h_1)$ is 
  \[
    \sigma_{C^*(G)/\ker\pi}(h_1 + \ker \pi)
    = \sigma_{C^*(G)/\ker\lambda}(h_1 + \ker \lambda)
    = \sigma(\lambda(h_1)) = \sigma_r,  
  \]
  whence $\pi(f(h_1)) = f(\pi(h_1)) = 0$ if and only if $f|_{\sigma_r} = 0$. 
\end{proof}

\begin{Proposition}\label{hspectrum}
  Let $-1 \leq t \leq 1$. In the orthogonal decomposition 
  $\pi_t(h_1) = t|\varsigma)(\varsigma| \oplus \pi_t(h_1)|_{\varsigma^\perp}$, the spectrum of $\pi_t(h_1)|_{\varsigma^\perp}$ is $\sigma_r$
  and therefore the spectrum of $\pi_t(h_1)$ is $\{t\} \cup \sigma_r$.
\end{Proposition}

\begin{proof}
  Let $\sigma$ ($\sigma'$) be the (essential) spectrum of $\pi_t(h_1)$ on $\sH$,
  i.e., $\sigma'$ consists of accumulation points of $\sigma$
  or eigenvalues of infinite multiplicity of $\pi_t(h_1)$.
  Since $\pi_t(h_1)$ for various $t$ coincide up to finite-rank operators,
  $\sigma'$ does not depend on $t \in [-1,1]$ (Weyl's stability theorem, see \cite{RS} XIII.4 for example).

By Proposition~\ref{prim},
$\ker\pi_t = \ker\lambda$ for $t \in \sigma_r$ and the spectrum of $\pi_t(h_1)$ in $\pi_t(C^*(G))$ 
is equal to that of $\lambda(h_1)$ in $C^*_r(G)$.
Consequently, if $t \in \sigma_r$, $\sigma = \sigma_r$ and hence $\sigma' = \sigma_r$.
% and we have $\sigma_r = \sigma' \subset \sigma$ for any $-1 \leq c \leq 1$.

Now $\sigma(\pi_t(h_1)|_{\varsigma^\perp}) \subset \sigma_r$ (Theorem~\ref{SP}) is combined with
$\sigma_r = \sigma' \subset \sigma(\pi_t(h_1)|_{\varsigma^\perp})$ to conclude that $\sigma(\pi_t(h_1)|_{\varsigma^\perp}) = \sigma_r$. 
% By the Plancherel theorem, the left regular representation $\lambda$ of $G$ is decomposed as
% \[
%   \lambda \cong \oint_{\sigma_r} \pi_c\, dc,
%   \quad
%   \tau^{1/2} \cong \oint_{\sigma_r} \epsilon_c^{1/2}\, dc 
% \]
% and hence, for almost all $c \in \sigma_r$, $\pi_c(\ker\lambda) = 0$,
% i.e., $\pi_c(C^*(G))$ is a quotient of $C^*_\text{red}(G)$. 
% Note here that $C^*(G)$ is separable and so is $\ker\lambda$. 
% 
% Thanks to the simplicity of $C^*_\text{red}(G)$,
% this implies $\ker\pi_c = \ker\lambda$ for almost all $c \in \sigma_r$.
% % this implies $\pi_c(C^*(G)) = C^*_\text{red}(G)$ 
% Then the spectrum $\sigma$ of $\pi_c(h_1)$ for such a $c$ is identical with that of $\lambda(h_1)$ on $\ell^2(G)$. 
% % i.e., $[-2\sqrt{r(1-r)},2\sqrt{r(1-r)}]$.
% Consequently
% $\sigma' = \sigma = \sigma_r$ 
% % for almost any $c$
% and we see that $\sigma_r \subset \sigma$ 
% % \[
% %   [-2\sqrt{r(1-r)},2\sqrt{r(1-r)}] \subset \sigma
% % \]
% for any $-1 \leq c \leq 1$.
% Moreover
% $\sigma(\pi_c(h_1)|_{\varsigma^\perp}) \subset \sigma_r$ (Theorem~\ref{SP}) is combined with
% $\sigma_r = \sigma' \subset \sigma(\pi_c(h_1)|_{\varsigma^\perp})$ to have $\sigma(\pi_c(h_1)|_{\varsigma^\perp}) = \sigma_r$. 
% % Since $c$ is an eigenvalue of $\pi_c(h_1)$ by Theorem~\ref{P2}, the assertion follows from this. 
\end{proof}

\begin{Proposition}\label{calkin}
  The continuous family $\pi_t$ ($-1 \leq t \leq 1$) of representations gives rise to the same *-homomorphism $C^*(G) \to \sB(\sH)/\sC(\sH)$
  by taking the quotient to the Calkin algebra $\sB(\sH)/\sC(\sH)$.
  % when $\pi_c:C^*(G) \to \sB(\sH)$ is composed with the quotient map to the Calkin algebra $\sB(\sH)/\sC(\sH)$.
  
  In particular, we have $\pi_t(\ker\pi) \subset \sC(\sH)$ for $-1 \leq t \leq 1$.
  Recall that $\pi = \pi_s$ with $s = {2\sqrt{r(1-r)}}$ and $\ker\pi_t = \ker\pi$ for $t \in \sigma_r$. 
\end{Proposition}

\begin{proof}
  Since $\pi_t(g) - \pi(g)$ is a finite-rank operator for each $g \in G$,
  the same holds for $\pi_t(x) - \pi(x)$ if $x \in \C G$.
  Since $x \in C^*(G)$ is norm-approximated by a sequence $x_n$ in $\C G$ and $\pi_t:C^*(G) \to \sB(\sH)$ is contractive in norm,
  $\lim_{n \to \infty}\| \pi_t(x_n) - \pi_t(x)\| = 0$ shows that $\pi_t(x) - \pi(x)$ is norm-approximated by
  a sequence $(\pi_t(x_n) - \pi(x_n))_{n \geq 1}$ of finite-rank operators.
  Thus $\pi_t(x) - \pi(x)$ is a compact operator on $\sH$. 
  % \[
  %   \lim_{n \to \infty} (\pi_c(x_n) - \pi(x_n))
  %   = \pi_c(x) - \pi(c) + \lim_{n \to \infty} (\pi_c(x_n) - \pi_c(x)) + \lim_{n \to \infty} (\pi(x) - \pi(x)_n)
  % \]
  % 
  % \[
  %   \pi_c(x) - \pi(x) = \pi_c(x) - \pi_c(x_n) + \pi_c(x_n) - \pi(x_n) + \pi(x_n) - \pi(x)
  % \]  
\end{proof}

\begin{Corollary}\label{calkin2} For $-1 \leq t \leq 1$, we have 
  $\ker\pi = \pi_t^{-1}(\sC(\sH)) \supset \ker \pi_t$ and hence $\pi_t(\ker\pi) = \pi_t(C^*(G)) \cap \sC(\sH)$.  
\end{Corollary}

\begin{proof}
 % Since $\pi_t(C^*(G)) + \sC(\sH)$ 
%  Since the irreducible representation $\pi_t$ is well-behaved for $t \in (-1,1) \setminus \sigma_r$ and
Since $\pi(C^*(G)) \cap \sC(\sH) = 0$ (Corollary~\ref{singularity}) and 
  \[
    \pi_t(C^*(G)) + \sC(\sH) = \pi(C^*(G)) + \sC(\sH), 
  \]
  % and then 
  % \[
  %  C^*(G)/\pi_t^{-1}(\sC(\sH)) \cong \pi_t(C^*(G))/\sC(\sH) \cong \pi(C^*(G)) \cong C^*(G)/ \ker\pi, %\pi(C^*(G)) \cap \sC(\sH) 
  % \]
  we see
   \begin{align*}
    C^*(G)/\pi_t^{-1}(\sC(\sH)) &\cong
                                  \pi_t(C^*(G))/ \bigl(\pi_t(C^*(G)) \cap \sC(\sH)\bigr)\\
    &\cong  \bigl(\pi_t(C^*(G)) + \sC(\sH)\bigr)/ \sC(\sH)\\
    &\cong \pi(C^*(G)) \cong C^*(G)/\ker\pi, 
  \end{align*}
  which means that $\ker\pi = \pi_t^{-1}(\sC(\sH)) \supset \ker\pi_t$.
   % Then $|\xi_e)(\xi_e| \in \pi_t(C^*(G))$ and the irreducibility of $\pi_t$ implies
\end{proof}

By the continuity of $\pi_t(g)$ % of the boundary representation $\pi = \pi_{2\sqrt{r(1-r)}}:C^*(G) \to \sB(\sH)$ % at $t = 2\sqrt{r(1-r)}$ 
in $t \in [-1,1]$ for each $g \in G$,
a norm-continuous family $(\pi_t(x))_{-1 \leq t \leq 1}$ of operators in $\sB(\sH)$ is associated to any $x \in C^*(G)$
as a uniform limit of continuous functions and 
the correspondence $x \mapsto (\pi_t(x))$ defines a *-homomorphism $\pi_*$ of $C^*(G)$ into
$C([-1,1])\otimes \sB(\sH)$ so that $\ker\pi_* = \bigcap_{-1 \leq t \leq 1} \ker \pi_t = \bigcap_{t \in (-1,1) \setminus \sigma_r} \ker \pi_t$ and
$\pi_*(\ker\pi) \subset C([-1,1])\otimes \sC(\sH)$ by Corollary~\ref{calkin2}.

In view of $\ker\pi_t = \ker\pi$ ($t \in \sigma_r$), $\pi_*(\ker\pi)$ is in fact
included in the ideal $C_{\sigma_r}([-1,1])\otimes \sC(\sH)$ of $C([-1,1])\otimes \sC(\sH)$, where 
\[
  C_{\sigma_r}([-1,1]) = \{ f \in C([-1,1]); f|_{\sigma_r} = 0\} \cong C_0([-1,1] \setminus \sigma_r)
\]
is the ideal of $C([-1,1])$ vanishing on $\sigma_r \subset [-1,1]$. 
% given by
% \[
%   \bigcap_{-1 \leq t \leq 1} \ker \pi_t = \bigcap_{t \in T} \ker \pi_t, 
% \]
% where $T = [-1,1] \setminus \sigma_r$. Recall that $\ker\pi_t = \ker\pi$ for $t \in \sigma_r$. 

% By the way of construction, 
% $\pi_t(g) - \pi(g)$ is a finite-rank operator on $\sH = L^2(\Omega)$ for any $g \in G$ and hence %$-1 \leq t \leq 1$ implies that
% $\pi_t(x) - \pi(x) \in \sC(\sH)$ for any $x \in C^*(G)$.
Recall also that $\pi_t$ is irreducible for $-1 < t < 1$,
whereas $\pi_{\pm 1}$ is decomposed as $\varepsilon_{\pm 1} \oplus \pi_\pm'$
according to the orthogonal decomposition $\sH = \C \varsigma \oplus \varsigma^\perp$.

Choose $f \in C_{\sigma_r}([-1,1])$ so that $f(\pm 1) = 1$.
By Theorem~\ref{SP} or Proposition~\ref{hspectrum}, % and Corollary~\ref{singularity},
$f(h_1) \in A$ satisfies $\pi_{\pm1}(f(h_1)) = |\varsigma)(\varsigma|$ and then
\[
  \pi_{\pm 1}\bigl(z f(h_1) + x(1-f(h_1))\bigr) = z|\varsigma)(\varsigma| + \pi_{\pm 1}(x) (1-|\varsigma)(\varsigma|)
  = z|\varsigma)(\varsigma| + \pi_{\pm}'(x)
\]
($z \in \C, x \in C^*(G)$) reveals that 
% Since $\varsigma$ is an eigenvector of $\pi_{\pm 1}(h_1)$ of an eigenvalue $\pm 1$ and
% the spectrum of $\pi_{\pm}'(h_1)$ is $\sigma_r$ by Proposition~\ref{hspectrum},
% $\varepsilon_{\pm 1}$ is disjoint from $\pi_{\pm}'$ as representations of $C^*(G)$,
% whence
$\pi_{\pm 1}(C^*(G)) = \C|\varsigma)(\varsigma| \oplus \pi_\pm'(C^*(G))$.

Consequently, Corollary~\ref{calkin2} is used again to see    
\begin{align*}
  \ker\pi &= \pi_{\pm 1}^{-1}(\sC(\sH)) = \pi_{\pm 1}^{-1}\Bigl(\C |\varsigma)(\varsigma| \oplus \pi_\pm'(C^*(G)) \cap \sC(\varsigma^\perp)\Bigr)\\
          &= (\pi_\pm')^{-1}(\sC(\varsigma^\perp)) \supset \ker \pi_\pm' 
\end{align*}
and hence $\ker\pi/\ker\pi_\pm' \cong \pi_\pm'(\ker\pi) = \pi_\pm'(C^*(G)) \cap \sC(\varsigma^\perp)$.

Define the residual representation of $C^*(G)$ on $\varsigma^\perp \oplus \varsigma^\perp$
to be a direct sum $\pi' = \pi_+' \oplus \pi_-'$, 
% Define the direct sum representation $\pi' = \pi_+' \oplus \pi_-'$ of $C^*(G)$ on $\xi_e^\perp \oplus \xi_e^\perp$,
which satisfies $\ker\pi' = \ker \pi_+' \cap \ker \pi_-' \subset \ker\pi$,
$\pi'(\ker\pi) \subset \pi'(C^*(G)) \cap \bigl(\sC(\varsigma^\perp) \oplus \sC(\varsigma^\perp)\bigr)$ and $\pi_*(\ker\pi') \subset C$ with
% \[
%   \pi'(x) \in \sC(\xi_e^\perp) \oplus \sC(\xi_e^\perp) \iff \pi_\pm'(x) \in \sC(\xi_e^\perp) \iff
%     x \in (\pi_\pm')^{-1}(\sC(\xi_e^\perp)) = \ker\pi.
%   \] 
\[
C = \{ f \in C_{\sigma_r}([-1,1])\otimes \sC(\sH); f(\pm 1) \in \C |\varsigma)(\varsigma|\}
\]
a C*-subalgebra of $C_{\sigma_r}([-1,1])\otimes \sC(\sH)$. % C for complementary and compact. 

% The middle equality is a consequence of $\pi'(\ker\pi) = \pi_+'(\ker\pi) \oplus \pi_-'(\ker\pi)$, which is checked as follows: 
% By choosing $(f(1),f(-1)) = (1,0)$ or $(0,1)$, $a_\pm = 1 - f(h_1) \in A$ satisfies
% $\pi_{\pm 1}(a_\pm) = 1 - |\varsigma)(\varsigma| = \pi_\pm'(a_\pm)$ and $\pi_{\mp 1}(a_\pm) = 0 = \pi_\mp'(a_\pm)$.
% Then, for $x_\pm \in C^*(G)$,
% $\pi'(x_+a_+ + x_-a_-) = \pi_+'(x_+) + \pi_-'(x_-)$. 

\begin{Proposition}\label{prim2}
  We have the equality $\pi_*(\ker\pi') = C$ and the associated isomorphism $C \cong \ker\pi'/\ker\pi_*$ induces
  a bijection between $\text{Prim}(C)$ and $\text{Prim}(\ker\pi') \setminus \text{Prim}(\ker\pi_*) \subset \text{Prim}(C^*(G))$
  so that $\text{Prim}(C)$ is the set of primitive ideals of spherical representations in the complementary series, i.e., 
  \[
  \{ \ker \pi_t; t \in (-1,1) \setminus \sigma_r \} \sqcup \{ \ker\varepsilon_{\pm 1}\}. 
  \]
\end{Proposition}

\begin{proof}
  We first describe pure states of $C$.
  Since a closed ideal of $C_0 = \{ f \in C; f(\pm 1) = 0\}$ of $C$ is isomorphic to $C_0((-1,1) \setminus \sigma_r)\otimes \sC(\sH)$ and
  satisfies $C/C_0 \cong \C \oplus \C$, irreducible representations of $C$ are given by evaluation at $t \in [-1,1] \setminus \sigma_r$ and hence
  pure states of $C$ are evaluations followed by applying pure states of fibers ($\sC(\sH)$ or $\C |\varsigma)(\varsigma|$).
  In particular, $C$ is a CCR (completely continuous representation) algebra. 

  The equality then follows from the Stone-Weierstrass theorem on CCR C*-algebras due to Kaplansky
  (see \cite{AA} for a survey) if 
  these pure states % together with the zero functional 
  are separated by elements in $\pi_*(\ker\pi')$.

  To see this, we recall Proposition~\ref{regularity} that
 % $\{ f(h_1) ; f \in C_{\sigma_r}[-1,1]\} \subset \ker\pi'$
  $\pi_t(f(h_1)) = f(t)|\varsigma)(\varsigma|$ holds for $f \in C_{\sigma_r}([-1,1])$ and $t\in [-1,1]$, whence $f(h_1) \in \ker \pi'$.
% shows that $\pi_t(\overline{C^*(G) f(h_1)C^*(G)}) = \sC(\sH)$ for $-1 < t < 1$ satisfying $f(t) \not= 0$.
  
  Let $\varphi$ and $\psi$ be
  pure states of $C$ given by $\varphi(x) = (\xi|x(s)\xi)$ and $\psi(x) = (\eta|x(t)\eta)$, 
  where $s,t \in [-1,1]\setminus \sigma_r$ and $\xi,\eta$ are unit vectors in $\sH$.

  If $s \not= t$ with $t \in (-1,1)$, $\pi_t$ is irreducible and we can find $y \in C^*(G)$ such that
  $\pi_t(y) = |\eta)(\varsigma|$ by Kadison's transitivity\footnote{We can dispense with Kadison's transitivity
   if $\eta$ is approximated by $\pi_t(y)\varsigma$.}.
  Then $yf(h_1)y^* \in \ker\pi'$ is represented by
  \[
    \pi_s(yf(h_1)y^*) = f(s) |\pi_s(y)\varsigma)(\pi_s(y)\varsigma|, \quad 
    \pi_t(yf(h_1)y^*) = f(t) |\eta)(\eta|.
  \]
  Thus, for a choice $f \in C_{\sigma_r}([-1,1])$ satisfying $f(s) = 0$ and $f(t) = 1$, $\varphi$ and $\psi$ are
  separated by $\pi_*(yf(h_1)y^*)$: 
  \begin{align*}
    \varphi(\pi_*(yf(h_1)y^*)) &= (\xi|\pi_s(yf(h_1)y^*)\xi) = f(s) (\xi|\pi_s(y)\varsigma) (\pi_s(y)\varsigma|\xi) = 0,\\
    \psi(\pi_*(yf(h_1)y^*)) &= (\eta|\pi_t(yf(h_1)y^*)\eta) = f(t) (\eta|\pi_t(y)\varsigma) (\pi_t(y)\varsigma|\eta) = 1. 
  \end{align*}
  
  When both $s\not= t$ come from $\{\pm 1\}$,
  we may assume that $s = -1$, $t=1$ with $\varphi(x) = (\varsigma|x(-1)\varsigma)$,
  $\psi(x) = (\varsigma|x(1)\varsigma)$ for $x = (x(t))_{-1 \leq t \leq 1} \in C$.
  Then $\varphi(\pi_*(f(h_1))) = f(-1)$ and $\psi(\pi_*(f(h_1))) = f(1)$ and we see that $\varphi$ and $\psi$ are separated by
  $f(h_1) \in \ker\pi'$ if $f \in C_{\sigma_r}([-1,1])$ satisfies $f(-1) \not= f(1)$.

  Finally consider the case $s=t$. Since states are unique for $t = \pm 1$,
  we assume $-1 < t < 1$ with $\varphi(x) = (\xi|x(t)\xi)$ and $\psi(x) = (\eta|x(t)\eta)$.
  Then the condition $\varphi \not= \psi$ is equivalent to $|(\xi|\eta)| < 1$ and $yf(h_1)y^* \in \ker\pi'$ described above is evaluated by
  \[   
    \varphi(\pi_*(yf(h_1)y^*)) = f(t) (\xi|\eta) (\eta|\xi),
    \quad 
    \psi(\pi_*(yf(h_1)y^*)) = f(t). 
  \]
  Thus $\varphi$ and $\psi$ are separated by $yf(h_1)y^* \in \ker\pi'$ if $f(t) \not= 0$.
  % Since each pure state of $D$ is extended to a pure state of $C[-1,1]\otimes \sC(\sH)$ and
  % irreducible representations of $C[-1,1]\otimes \sC(\sH)$ are given by evaluations, 
\end{proof}

% We now restrict $\pi_t$ to a cosimple ideal $\ker\pi$ ($C^*(G)/\ker\pi = C^*_r(G)$ being simple) and observe that
% $\pi_t(x) \in \sC(\sH)$ ($x \in \ker\pi$).
% In view of $\ker\pi_t = \ker\pi$ ($t \in \sigma_r$), $\pi_t(\ker\pi) = \{0\}$ if $t \in \sigma_t$.
% % For $t \in [-1,1] \setminus \sigma_r$,
% Since the spectral property of $\pi_t(h_1)$ with the functional calculus of $h_1$ gives
% $\pi_t(f(h_1)) = f(t)|\xi_e)(\xi_e|$ if $f \in C[-1,1]$ vanishes on $\sigma_r$ (i,e., $f(h_1) \in \ker\pi$),
% we see $|\xi_e)(\xi_e| \in \pi_t(\ker\pi)$ for $t \in [-1,1] \setminus\sigma_r$ and, if $t \in (-1,1)\setminus\sigma_r$, 
% the irreducibility of $\pi_t$ implies $\pi_t(\ker\pi) = \sC(\sH)$
% and $\pi_t$ is well-behaved in the sense that $\sC(\sH) \subset \pi_t(C^*(G))$. 

\begin{Lemma}
We have $\ker\pi = \ker\pi'$.   
\end{Lemma}

\begin{proof}
% It seems that we have found the coincidence $\ker\pi = \ker\pi'$.
  To see this, we use the continuous deformation $(\lambda_\upsilon)_{-1 \leq \upsilon \leq 1}$ of the regular representation $\lambda = \lambda_0$ in \cite{PS}. 
  % due to Pytlik and Szwarc.
  Recall that these are unitary representations of $G$ with the following properties (Theorem~\ref{SS}): 
  % (Proposition~\ref{unitaryequivalence}):
\begin{enumerate}
  \item
$\lambda_\upsilon$ is unitarily equivalent to $\lambda$ if $\upsilon \in \sigma_r$, 
to $\pi_\upsilon \oplus \lambda$ if $\upsilon \in (-1,1) \setminus \sigma_r$
and to $\epsilon_{\pm 1} \oplus \overbrace{\lambda \oplus \dots \oplus \lambda}^{\text{$l$-times}}$ if $\upsilon = \pm 1$.
\item
  $\ker \lambda_\upsilon = \ker\pi_{c_r(\upsilon)} \cap \ker \lambda = \ker \pi_{c_r(\upsilon)} \cap \ker\pi = \ker \pi_{c_r(\upsilon)}$ if $|\upsilon| > \sqrt{r/(1-r)}$,
  $\ker\lambda_\upsilon = \ker\lambda$ if $|\upsilon| \leq \sqrt{r/(1-r)}$, and 
  $\ker\lambda_{\pm 1} = \ker\epsilon_{\pm 1} \cap \ker\lambda = \ker \epsilon_{\pm 1} \cap \ker\pi$.
  Notice that $c_r(\upsilon) \not\in \sigma_r$ for $\upsilon \not= \pm \sqrt{r/(1-r)}$. 
\end{enumerate}
($\pi_\pm'$ would be unitarily equivalent to $\lambda \otimes 1_{\C^{l-1}}$ but its validity is irrelevant here.)
% but this is not relevant to see $\ker\pi = \ker\pi'$.

As observed for $(\pi_t)$ before, the family $(\lambda_\upsilon)_{-1 \leq \upsilon \leq 1}$ of *-representations of $C^*(G)$ satisfies
$\lambda_\upsilon(x) - \lambda(x) \in \sC(\ell^2(G))$ for $x \in C^*(G)$ and gives rise to
a *-homomorphism $\lambda_*: C^*(G) \to C([-1,1]) \otimes \sB(\ell^2(G))$ in such a way that
$\lambda_*(\ker\lambda) = \lambda_*(\ker\pi) \subset C([-1,1])\otimes \sC(\ell^2(G))$.
In view of norm-continuity of $\lambda_\upsilon(x)$ in $\upsilon \in [-1,1]$ and Theorem~\ref{SS}, 
\[
  \ker \lambda_* = \bigcap_{\sqrt{r/(1-r)} < |\upsilon| < 1} \ker \lambda_\upsilon
  = \bigcap_{4r(1-r) < t^2 < 1} \ker \pi_t
  % = \bigcap_{4r(1-r) < t^2 < 1} \ker \pi_t 
  % = \bigcap_{t \in (-1,1) \setminus \sigma_r} \ker \pi_t
  = \ker \pi_*
\]
and $\lambda_*(\ker\lambda) = \lambda_*(\ker\pi) \subset C([-1,1]) \otimes \sC(\ell^2(G))$.
Note here that $\{ c_r(\upsilon); \sqrt{r/(1-r)} < |\upsilon| < 1\} = \{ t \in \R; 4r(1-r) < t^2 < 1\}$.

Thus pure states of $\ker\pi/\ker\pi_* = \ker\lambda/\ker\lambda_* \cong \lambda_*(\ker\lambda)$
are given by pure states of $\lambda_\upsilon(\ker\lambda)$ 
after the evaluation at $\upsilon$ satisfying $|\upsilon| > \sqrt{r/(1-r)}$. Here recall that
\[
  \lambda_\upsilon(\ker\lambda) \cong \begin{cases} 0 &(|\upsilon| \leq \sqrt{r/(1-r)})\\
                                        \pi_{c_r(\upsilon)}(\ker\pi) &(\sqrt{r/(1-r)}  < |\upsilon| < 1)\\
                                 \C &(\upsilon = \pm 1)
                               \end{cases}. 
\]

Consequently the set of associated primitive ideals of $\ker\lambda/\ker\lambda_* = \ker\pi/\ker\pi_*$ is 
  \[
    \{ \ker \pi_t; t \in (-1,1) \setminus \sigma_r \} \cup \{ \ker\epsilon_{\pm 1}\}. 
  \]
  At this point, there might be overlapping in the union but the comparison of this with Proposition~\ref{prim2} enables us
  to conclude that these are in fact distinct and $\ker\pi/\ker\pi' = 0$.
  % $R \cong \ker \pi'/\ker \pi_* \subset \ker\pi/\ker \pi_*$ 
\end{proof}

{\small
\begin{Remark}
We have an inclusion $\overline{C^*(G)f(h_1) C^*(G)} \subset \ker\pi$ for $f \in C_{\sigma_r}([-1,1])$ and 
the equality $\ker\pi' = \ker\pi$ is rephrased by 
\[
\ker\pi = \overline{\bigcup_{f \uparrow [-1,1]\setminus \sigma_r} C^*(G)f(h_1) C^*(G)}. 
\]
Here $f \uparrow [-1,1]\setminus \sigma_r$ means an inductive limit on $f \in C_{\sigma_r}([-1,1])$ satisfying $0 \leq f \leq 1$.
% It seems difficult to check this purely in terms of the Pytlik deformation. 
\end{Remark}}

As a summary of consideration so far, 

\begin{Theorem}
  Primitive ideals of the radial C*-algebra $C^*_{rad}(G)$ are exactly kernels of spherical representations of $G$:
  \[
    \text{Prim}(C^*_{rad}(G)) = \{ \ker\pi\} \sqcup \{ \ker \pi_t; t \in (-1,1) \setminus \sigma_r \} \sqcup \{ \ker\varepsilon_{\pm 1}\}
  \]
  with the primitive ideal space of $\ker \lambda/\ker \lambda_* = \ker\pi/\ker\pi_*$ identified with
$\{ \ker \pi_t; t \in (-1,1) \setminus \sigma_r \} \sqcup \{ \ker\varepsilon_{\pm 1}\}$. 
\end{Theorem}

We now look into the topology of the primitive ideal space $\Delta \equiv \text{Prim}(C^*_{rad}(G))$,
which is a closed subset of $\text{Prim}(C^*(G))$.
Since primitive ideals in $\Delta$ are of the form $[t] = \ker\pi_t$ ($-1 < t < 1$) or $[\pm 1] = \ker\varepsilon_{\pm 1}$,
$\Delta$ is identified with a quotient of $[-1,1]$ in such a way that $\sigma_r$ is shrunken to one point in $\Delta$.
% primitive ideals related with radial representations. 
% Let $R$ be the set of primitive ideals associated to irreducible radial representations, which is furnished
% with the relative topology from the primitive ideal space of $C^*(G)$.
% Since irreducible radial representations are parametrized by the spectral pamaeter $-1 \leq c \leq 1$,
% $R = \{ \ker \pi_c; -1 < c < 1\} \sqcup \{\ker\epsilon_{\pm 1}\}$, 
% the C*-algebra generated by the radial representation of $G$ on $\sR$: 
% \[
%   R = C^*(G)/ \bigcap_{-1 \leq c \leq 1} \ker \pi_c
% \]
% as a quotient algebra of $C^*(G)$ and $\text{Prim}(R) = \{ \ker \pi_c; -1 \leq c \leq 1\}$,
% To describe this, introduce a notation 
% \[
%   [t] =
%   \begin{cases}
%     \ker\pi_t &(-1 < t < 1)\\
%     \ker\epsilon_t &(t=\pm 1)
%   \end{cases} 
% \]
% so that the quotient map $[-1,1] \to R$ is given by $t \mapsto [t]$.

We first check the continuity of $[t] \in \Delta$ in $t \in [-1,1]$: 
If a sequence $(t_n)$ in $[-1,1]$ converges to $t \in [-1,1]$, then  
$\bigcap_{n \geq 1} \ker \pi_{t_n} \subset \ker \pi_t$ because 
$x \in \bigcap_{n \geq 1} \ker \pi_{t_n}$ and $y \in C^*(G)$ satisfy 
\[
  \| \pi_t(xy)\varsigma\| = \| xy \epsilon_t^{1/2}\| = \lim_{n \to \infty} \| xy\epsilon_{t_n}^{1/2} \|
  = \lim_{n \to \infty} \| \pi_{t_n}(x) \pi_{t_n}(y) \varsigma \| = 0.
\]

\begin{Theorem}\label{closure}
For a non-empty subset $T$ of $(-1,1) \setminus \sigma_r$, we have 
\[
  \overline{\{[t]; t \in T \}} = \{ [t]; t \in \overline{T}\} \cup \{ \ker\pi\}
\]
in $\Delta = \text{Prim}(C^*_\text{rad}(G))$.
Here the left hand side is the closure in $\Delta$ and $\overline{T}$ denotes the closure of $T$ in $[-1,1]$.

Since $\ker \varepsilon_{\pm 1}$ and $\ker\pi$ are maximal ideals, they are closed as one-point sets.
Thus the closure operation in $\Delta$ is completely described by this.
\end{Theorem}

\begin{proof}
Let us begin with showing that $\Delta$ contains relatively open intervals in $[-1,1] \setminus \sigma_r$ as open subsets of $\Delta$: 
For a continuous function $f \in C([-1,1])$, Proposition~\ref{hspectrum} gives
\[
  \| \pi_t(f(h_1))\| =
  \begin{cases}
    \max\{ |f(t')|; t' \in \{ t\} \cup \sigma_r\} &(-1 < t < 1),\\
    |\varepsilon_t(f(h_1))| = |f(\pm 1)| &(t = \pm 1). 
  \end{cases}
\]
Since relatively open subsets of $[-1,1] \setminus \sigma_r$ are then realized in the form
\[
  \{ t \in [-1,1] \setminus \sigma_r; \| \pi_t(f(h_1)) \| > 0 \}
\]
with $f = 0$ on $\sigma_r$ (relatively open subsets being disjoint unions of countably many relatively open intervals),
they are open in $\Delta$ as well
because $\Delta \ni [t] \mapsto \| \pi_t(f(h_1)\|$
is a lower semicontinuous function (\cite{Di} Proposition 3.3.2). 

Thus a relatively open subset $[-1,1] \setminus (\sigma_r \cup \overline{T})$
of $[-1,1] \setminus \sigma_r$ is open in $\Delta$ and hence
its complement $[\overline{T}] \cup \{ \ker\pi\}$ in $\Delta$ is closed. Consequently % One then sees that 
\[
 [\overline{T}] \subset \overline{[T]} \subset [\overline{T}] \cup \{ \ker\pi\}.
% \{ \ker\pi_c; c \in \overline{C}\} \subset \overline{\{\ker\pi_c; c \in C\}} \subset \{ \ker \pi_c; c \in \overline{C}\} \cup \{ \ker\pi\}.
\]
Here the first inclusion is due to the continuity of $[t]$ in $t \in [-1,1]$. 
Since $T$ is non-empty, its closure $\overline{[T]}$ in $\Delta$ contains $\ker\pi$ (Corollary~\ref{calkin2}) and the assertion is proved.
\end{proof}

\begin{Corollary}
  Open sets of $\Delta$ are exactly of the following form:
  \begin{enumerate}
    \item
      $[U]$ with $U \subset [-1,1] \setminus \sigma_r$ an open subset in the relative topology of $[-1,1]$. 
    \item
      $\Delta \setminus F$ with $F$ a subset (including the empty set) of $\{ [1], [-1]\}$.
   \end{enumerate}
\end{Corollary}

\begin{Corollary}
% The representation image of $C^*(G)$ on $\sR$
  The radial C*-algebra $C^*_{\text{rad}}(G)$ of $G$ is isomorphic to $C^*(G)/\ker\pi_*$ and 
the center of $C^*_{\text{rad}}(G)$ is trivial.
%  The center of $C^*(G)/\bigcap_{J \in R} J$ is trivial. 
\end{Corollary}

\begin{proof}
  % Let $\overline{\Delta}$ be the closure of $\Delta$ in $\text{Prim}(C^*(G))$ and
  % $f: \overline{\Delta} \to \C$ be a continuous function.
  Let $f: \Delta \to \C$ be a continuous function.
  Since $\ker \pi \in \overline{\{\ker \pi_t\}}$ ($-1 < t < 1$),
  $f([t]) = f(\ker\pi)$ for $-1 < t < 1$ and then for $-1 \leq t \leq 1$ by continuity of $[t]$ in $t$. % in $c \in [-1,1]$.
Thus $f$ is constant on $\Delta$ and % so is on $\overline{\Delta}$ again by continuity and 
  the assertion follows from the Dauns-Hofmann theorem (\cite{Pe} \S 4.4). 
\end{proof}

\end{document}